\documentclass{article}

\usepackage[utf8]{inputenc}
\usepackage[all,arc]{xy}
\usepackage{enumerate}
\usepackage[top=1in, bottom=1.25in, left=1.25in, right=1.25in]{geometry} 
\usepackage{parskip}
\usepackage{braket}
\usepackage{tikz-cd}
\usepackage{verbatim}
\usepackage{tabu}
\usepackage{ggmaths}

\usepackage{biblatex}
\usepackage{hyperref}
\usepackage{xurl}
\hypersetup{breaklinks=true}
\newcommand{\F}{\mathcal{F}}

\title{Partial results for union-closed conjectures on the weighted cube}
\author{\makeauthorentry}

\begin{document}

\maketitle

\begin{abstract}

The celebrated union-closed conjecture is concerned with the cardinalities of various subsets of the Boolean $d$-cube. The cardinality of such a set is equivalent, up to a constant, to its measure under the uniform distribution, so we can pose more general conjectures by choosing a different probability distribution on the cube. In particular, for any sequence of probabilities $(p_i)_{i=1}^d$ we can consider the product of $d$ independent Bernoulli random variables, with success probabilities $p_i$. In this short note, we find a generalised form of Karpas' \cite{karpas} special case of the union-closed conjecture for families $\F$ with density at least half. We also generalise Knill's \cite{knill} logarithmic lower bound.
\end{abstract}

\section{Introduction}

For a natural number $d$ we consider the ground set $[d] = \set{1, 2, \dots, d}$. Let $\F$ be a family of subsets of the ground set; that is, $\F \subset \mathcal{P}([d])$. Equivalently, we let $\Omega_d = \set{0,1}^d$ be the Boolean cube of dimension $d$, and consider $\F \subset \Omega_d$ a subset of its vertices; then $A \subset [d]$ corresponds to $\vec{a} \in \Omega_d$ by $\vec{a} = \mathbbm{1}_A$ and $A = \supp(\vec{a})$. The family $\F$ is said to be union-closed if for any $A, B \in \F$ we have $A \cup B \in \F$ as well.

In the 1970s, P\'eter Frankl posed the following conjecture:

\begin{conj}
Let $\F$ be a union closed family in $\mathcal{P}([d])$. For $i \in [d]$, let $\F_i$ be the subfamily of $\F$ consisting of sets containing $i$. That is $\F_i = \set{A \in \F | i \in A}$. Then there exists $i \in [d]$ such that $|\F_i|/|\F| \geq 1/2$.
\end{conj}

Note that $\F_i = \set{x \in \F | x_i=1}$.

This problem is famous for its deceptively simple statement, and for its intractability. In 2022, Gilmer \cite{gilmer} finally made significant progress, giving a constant lower bound $0 < \pi < 1/2$ for $\max_i |\F_i|/|\F|$. The value of $\pi$ was improved in \cite{sawin} and then \cite{yu}, but substantially new ideas would be required to get to $\pi=1/2$.

On the other hand, Karpas \cite{karpas} gave a short proof in 2017 that the full conjecture holds whenever $\F$ has density at least $1/2$.

\begin{thm}[Karpas]\label{thm:karpas}
Let $\F$ be a union-closed family in $\mathcal{P}([d])$ with $|\F| \geq 2^{d-1}$. Then there exists $i \in [d]$ such that $|\F_i|/|\F| \geq 1/2$.
\end{thm}

Karpas employed a simple argument from the analysis of Boolean functions. In Section \ref{section:karpas} we extend this argument to the $\vec{p}$-weighted cube; that is a random variable $X \in \Omega_d$ given by $X_i \sim B(1,p_i)$ with independence between the $d$ dimensions. We say that $\mu = \mu_{\vec{p}}$ is the probability distribution for the value of $X$, so $\mu(\{A\}) =  \prod_{i \in A} p_i \prod_{i \notin A} q_i$, where $q_i = 1-p_i$. We also write $p = \min_i p_i$ and $q=1-p$. Note that putting $p_i \equiv 1/2$, we get the uniform distribution back.

In Theorem \ref{thm:weightedkarpas} we prove a $\vec{p}$-weighted version of Theorem \ref{thm:karpas}, namely that whenever $\mu(\F) \geq q$, there is some $i$ for which $\mu(\F_i) \geq q_i \mu(\F)$.

Knill's \cite{knill} result, although now superceded by that of Gilmer \cite{gilmer}, provided a lower bound on the size of some $\F_i$; where the conjecture claims that we can always find $i$ such that $|\F_i|/|\F| \geq 1/2$, Knill showed that at least we can find $i$ such that $|\F_i|/|\F|$ is bounded below by a logarithmic term. We assume that $\emptyset \in \F$, noting that this can only decrease each $|\F_i|/|\F|$.

\begin{thm}[Knill]\label{logboundcor}
There exists $i \in [d]$ such that $|\F_i| \geq \frac{|\F|-1}{\log_2 |\F|}$
\end{thm}

In Section \ref{section:knill} we follow the same argument, concluding in Theorem \ref{logboundpcor} with a set of lower bounds for $\mu(\F_i)/\mu(\F)$ at least one of which must be true for each $\F$.

\section{Families with large weight}\label{section:karpas}

In this section we generalise Karpas' \cite{karpas} result to the case of the $\vec{p}$-weighted cube, for any $(p_i)_i \in [0,1]$. For functions $f, g: \Omega_d \rightarrow \mathbb{R}$ we define the usual inner product $\innerproduct{f}{g} = \mu(fg) = \expect[x \sim X]{f(x)g(x)}$. The space of functions under this inner product has the following orthonormal basis. For $1 \leq i \leq d$ define

\begin{equation}
\chi_i(x) = \alpha_i(x_i - p_i)
\end{equation}

where

\begin{equation}
\alpha_i = 1/\sqrt{p_i q_i}
\end{equation}

Then for $S \subset [d]$ define $\chi_S = \prod _{i \in S} \chi_i$. Now the fourier coefficients of a function $f: \Omega_d \rightarrow \mathbb{R}$ are given by $\hat{f}(S) = \innerproduct{f}{\chi_S}$, and the level $k$ weights of $f$ are given by $W^k(f) = \sum_{|S|=k} \hat{f}(S)^2$.

Parseval's identity gives $\sum_{k=0}^n W^k(f) = \mu(f^2)$. We define $i$th influences in the usual way:

\begin{defn}[Influence]
For $f: \Omega \rightarrow \set{-1,1}$, the $i$th positive influence, negative influence, and influence are defined respectively as

\begin{equation}
\begin{aligned}
I_i^+(f) &= \prob[x \sim X]{f(x_{i \rightarrow 0}) = -1, f(x_{i \rightarrow 1}) = 1} \\
I_i^-(f) &= \prob[x \sim X]{f(x_{i \rightarrow 0}) = 1, f(x_{i \rightarrow 1}) = -1} \\
I_i(f) &= I_i^+(f) + I_i^-(f)
\end{aligned}
\end{equation}

(Note that for $x$ a vector indexed by set $I$, $J \subset I$, and $w$ a vector indexed by set $J$, we write $x_{J \rightarrow w}$ for the vector indexed by $I$, agreeing with $x$ on $I\setminus J$ and with $w$ on $J$.)

Moreover, the (total) influence of $f$ is $I(f) = \sum_i \frac{4}{\alpha_i^2} I_i(f)$, and similarly for $I^+(f)$ and $I^-(f)$.

\end{defn}

We make the following observations:

\begin{prop}\label{lowdegree}
Let $f: \Omega \rightarrow \set{-1,1}$ be a boolean function. Then:
\begin{enumerate}
\item $\hat{f}(\emptyset) = 2\mu(f^{-1}(1)) - 1$
\item $\alpha_i \hat{f}(\{i\}) = 2(I_i^+(f) - I_i^-(f))$
\end{enumerate}
\end{prop}

These follow immediately from evaluating $\hat{f}$ at $\emptyset$ and at all singleton sets.

\begin{prop}\label{influence}
Let $f: \Omega \rightarrow \set{-1,1}$. Then

\begin{equation}
I(f) = \sum_{k=0}^n kW^k(f)
\end{equation}
\end{prop}

\begin{proof}
This follows from the following equation for all $i$:

\begin{equation}
I_i(f) = \frac{\alpha_i^2}{4}\sum_{i \in S} \hat{f}(S)^2
\end{equation}

To establish this, we follow O'Donnell in \cite{odonnell}. We define $D_i$ to be a linear operator such that $D_i f(x) = \frac{f(x_{i \rightarrow 1}) - f(x_{i \rightarrow 0})}{2}$. In a sense this represents the partial derivative of $f$ in the variable $x_i$. Now 

\begin{equation}
D_i f(x) = \frac{\alpha_i}{2} \sum_{i \in S} \hat{f}(S) \chi_{S \setminus \{i\}}(x)
\end{equation}

so by Parseval,

\begin{equation}
\mu((D_i f(x))^2) = \frac{\alpha_i^2}{4}\sum_{i \in S} \hat{f}(S)^2
\end{equation}

On the other hand, $D_i f(x)$ takes one of the values $\pm 1$ precisely when it contributes to $I_i(f)$, so $I_i(f) = \mu((D_i f(x))^2)$ and the proof is complete.
\end{proof}

Combining Proposition \ref{influence} with Parseval's identity on $f$, noting that $f^2 \equiv 1$ in this case, we get the following corollary:

\begin{cor}\label{lowdegineq}
For any $1 \leq k \leq d$,

\begin{equation}
I(f) \geq \left(k - \sum_{i=0}^{k-1}(k-i)W^i(f)\right)
\end{equation}
\end{cor}

Now we prove our weighted version of Karpas' Theorem \ref{thm:karpas}:

\begin{thm}
Let $\F \subset \Omega$ be simply rooted, with $\mu(\F) \leq p$. Then there is an element $1 \leq i \leq d$ such that $\mu(\F_i) \leq p_i \mu(\F)$.
\end{thm}

Recall that a simply rooted family is the complement of a union-closed family. Then for $x \in \F$, there exists at most one value $i \in [d]$ for which $x_{i \rightarrow 0} \notin \F$. Indeed, if $i \neq j$ both satisfy this condition then $x_{i \rightarrow 0}$ and $x_{j \rightarrow 0}$ are in $\F^c$, a union-closed family; then $x = x_{i \rightarrow 0} \vee x_{j \rightarrow 0}$ is also in $\F^c$, a contradiction.

\begin{proof}
Let $\mu(\F) = p - \delta$ with $\delta \geq 0$. Assume for a contradiction that $\mu(\F_i) > p_i \mu(\F)$ for all $i$. Let $f$ be the indicator for $\F$, taking the value $1$ on $\F$ and $-1$ elsewhere. Then

\begin{equation}\label{eq:positiveinfluence}
\begin{aligned}
\hat{f}(\{i\}) &= q_i \alpha_i (2\mu(\F_i) - p_i) -p_i \alpha_i (2\mu(\F \setminus \F_i) - q_i)\\
&= 2\alpha_i(\mu(\F_i) - p_i\mu(\F)) \\
&>0
\end{aligned}
\end{equation}

First we give an upper bound for $I^+(f)$. For all $x \in \F$ there is at most one $i \in [d]$ with $i \in x$ and $f(x_{i \rightarrow 0}) = -1$. Thus each $x \in \F$ contributes at most $\mu(x) + \mu(x_{i \rightarrow 0}) = \mu(x)/p_i$ to $I^+(f)$, corresponding to a contribution of $4q_i \mu(x)$ to the total positive influence, so $I^+(f) \leq 4q\mu(\F) = 4q(p-\delta)$.

As for a lower bound for $I^+(f)$, we know from Proposition \ref{lowdegree} that

\begin{equation}
\hat{f}(\emptyset) = 2(p-\delta) -1
\end{equation}

Also by Proposition \ref{lowdegree}, we have $\alpha_i \hat{f}(\{i\}) = 2(I_i^+(f) - I_i^-(f))$. The left hand side is positive by equation \ref{eq:positiveinfluence}, but the right hand side is less than $2$ trivially (to get equality, we would require $\F$ to be $\set{x: x_i = 1}$, which is a trivial case). Thus

\begin{equation}
\begin{aligned}
0 < \alpha_i^2 /4 \hat{f}(\{i\})^2 &< \frac{\alpha_i}{2} \hat{f}(\{i\}) < 1 \\
\hat{f}(\{i\})^2 &< 2/\alpha_i \hat{f}(\{i\}) = \frac{4}{\alpha_i^2} (I_i^+(f)-I_i^-(f)) \\
W^1(f) &< I^+(f)-I^-(f)
\end{aligned}
\end{equation}

Now by Corollary \ref{lowdegineq} with $k=2$ we get

\begin{equation}
I(f) \geq 2 - 2\hat{f}(\emptyset)^2 - W^1(f)
\end{equation}

So

\begin{equation}
I^+(f)+I^-(f) > 2 - 2(2(p-\delta)-1)^2 - I^+(f)+I^-(f)
\end{equation}

Thus we get the lower bound

\begin{equation}
I^+(f) > 4(p-\delta)(q+\delta)
\end{equation}

Combining with the upper bound gives

\begin{equation}
4q(p-\delta) > 4(p-\delta)(q+\delta)
\end{equation}

which simplifies (given $\delta < p$, i.e. $\F$ is non-empty) to $q > q+\delta$, a contradiction.
\end{proof}

Taking the complements of $\F$ and $\F_i$ gives a statement about union-closed families:

\begin{thm}\label{thm:weightedkarpas}
Let $\F \subset \Omega$ be union-closed, with $\mu(\F) \geq q$. Then there is an element $i \in [d]$ such that $\mu(\F_i) \geq q_i \mu(\F)$.
\end{thm}

\section{A logarithmic bound for the weighted cube}\label{section:knill}

In this section we generalise Knill's \cite{knill} result to the case of the $\vec{p}$-weighted cube, where for $x \sim X_{\vec{p}}$ we have $x_i=1$ independently with probability $p_i$. A hitting set of $\F$ is a set $S \subset [d]$ such that every non-empty $A \in \F$ has non-empty intersection with $S$. Recall that we assume $\emptyset \in \F$.

Knill \cite{knill} proved the following:

\begin{thm}[Knill]\label{logbound} For all hitting sets which are minimal under the inclusion ordering we have $|S| \leq \log_2 |\F|$.
\end{thm}

\begin{proof}
For each $s \in S$ there exists $A_s \in \F$ such that $A_s \cap S = \set{s}$. Otherwise $S \setminus \set{s}$ is a hitting set so $S$ was not minimal. Now for all $\emptyset \neq T \subset S$ we define $A_T = \cup_{s \in T} A_s$. By the union-closed property $A_T \in \F$. Moreover $A_T \cap S = T$ so each $A_T$ is distinct. Hence we have found an injection from the non-empty subsets of $S$ into non-empty elements of $\F$. There are $2^{|S|}-1$ such subsets so $|\F|-1 \geq 2^{|S|}-1$ as required.
\end{proof}

As an immediate corollary we get Theorem \ref{logboundcor}.

\begin{proof}[Proof of Theorem \ref{logboundcor}]
We know that $|S| \leq \log_2 |\F|$. Each of the non-empty elements of $\F$ contains some $s \in S$, so $\sum_s |\F_s| \geq |\F|-1$. Then for some $s \in S$ we have $|\F_s| \geq \frac{|\F|-1}{|S|} \geq \frac{|\F|-1}{\log_2 |\F|}$.
\end{proof}

Both Theorem \ref{logbound} and Corollary \ref{logboundcor} can be interpreted as statements about $\mu_{1/2}(\F)$. What happens when we consider $\mu_{\vec{p}}$ instead? We restrict to the case of $p_i \geq 1/2$ for all $i$. We also define $q_i = 1-p_i$ and $Q_i = 1/q_i$. We also abbreviate products of these values over a set $R$ by $p_R$, $q_R$ and $Q_R$, and we write $Q = Q_{[d]}$. For convenience we will simply write $\mu$ instead of $\mu_{\vec{p}}$.

\begin{thm}\label{logboundp}
For all hitting sets which are minimal under the inclusion ordering we have

\begin{align}\label{eq:logboundp}
\sum_{i \in S} \log Q_i \leq \log(Q\mu(\F))
\end{align}
\end{thm}

Note that in the case of $p_i \equiv 1/2$ we get back Theorem \ref{logbound}.

The corresponding corollary to Theorem \ref{logboundcor} is

\begin{cor}\label{logboundpcor}
There exists $i \in [d]$ such that

\begin{equation}
\frac{\mu(\F) - 1/Q}{\log (Q\mu(\F))} \leq\frac{\mu(\F_i)}{\log Q_i}
\end{equation}
\end{cor}

The proof follows the same argument as above.

\begin{proof}[Proof of Theorem \ref{logboundp}]
As above, we can find distinct sets $A_T$ in $\F$ for each $\emptyset \neq T \subset S$, with $T = A_T \cap S$. Now 

\begin{equation}
\begin{aligned}
\mu(\{A_T\}) &= p_{A_T}q_{A_T^c}\\
&= p_T q_{T^c} \prod_{i \in A_T \setminus T} p_i/q_i\\
&\geq p_T q_{T^c}
\end{aligned}
\end{equation}

The inequality holds because $p_i \geq 1/2$ so $p_i \geq 1-p_i = q_i$. We let $\mathcal{A} = \set{A_T | \emptyset \neq T \subset S}$; then $\mathcal{A} \subset \F \setminus \{\emptyset\}$. Moreover, $\emptyset \in \F$ so $\mu(\F \setminus \{\emptyset\}) = \mu(\F) - \mu(\{\emptyset\}) = \mu(\F) - q_{[d]}$. Then

\begin{equation}
\begin{aligned}
\mu(\F) - q_{[d]} = \mu(\F \setminus \{\emptyset\}) &\geq \mu(\mathcal{A}) \\
&= \sum_{\emptyset \neq T \subset S} \mu(\{A_T\}) \\
&\geq \sum_{\emptyset \neq T \subset S} p_T q_{T^c}
\end{aligned}
\end{equation}

So

\begin{equation}
\begin{aligned}
\mu(\F) &\geq \sum_{T \subset S} p_T q_{T^c} \\
&= q_{S^c}
\end{aligned}
\end{equation}

Multiplying both sides by $Q$ and taking logarithms gives our result.
\end{proof}

\begin{proof}[Proof of Corollary \ref{logboundpcor}]

Since $S$ is a hitting set we know that 

\begin{equation}
\begin{aligned}
\mu(\F) - q_{[d]} = \mu(\F \setminus \{\emptyset\}) &\leq \sum_{i \in S} \mu(\F_i) \\
&= \sum_{i \in S} \log Q_i \cdot \frac{\mu(\F_i)}{\log Q_i}
\end{aligned}
\end{equation}

Dividing by the sides of the inequality \ref{eq:logboundp}, this becomes

\begin{equation}
\frac{\mu(\F) - q_{[d]}}{\log (Q\mu(\F))} \leq \frac{\sum_{i \in S} \log Q_i \cdot \frac{\mu(\F_i)}{\log Q_i}}{\sum_{i \in S} \log Q_i}
\end{equation}

The right hand side is a weighted mean so there exists an $i$ such that

\begin{equation}
\frac{\mu(\F) - 1/Q}{\log (Q\mu(\F))} \leq\frac{\mu(\F_i)}{\log Q_i}
\end{equation}

This is the bound required.
\end{proof}

\section{Acknowledgements}

I am grateful to Gil Kalai for introducing me to Ilan Karpas' result and for encouraging me to look for Theorem \ref{thm:weightedkarpas} at the very beginning of my time studying with him, in 2018. Thanks also to Igor Balla who spoke enthusiastically about Knill's log bound, giving me the idea to apply the same generalisation to that result.

\printbibliography

\end{document}